\documentclass[twoside,12pt]{article}
\usepackage{amsbsy,latexsym,amsopn,amstext,amsxtra,euscript,amscd}
\usepackage{amssymb,amsfonts,amsmath,amsthm}
\usepackage{graphics,graphicx}

\graphicspath{{Fig/}}

\newtheorem{theorem}{Theorem}
\newtheorem{corollary}{Corollary}
\newtheorem{identity}{Identity}[theorem]

\newcommand*{\ADRt}{School of Mathematics and Statistics, Wuhan University, Wuhan, China. \texttt{komatsu@whu.edu.cn}} 
\newcommand*{\ADRn}{University of Sopron,  Institute of Mathematics, Hungary. \texttt{nemeth.laszlo@uni-sopron.hu}}
\newcommand*{\ADRs}{{University J.~Selye, Department of Mathematics and Informatics, Slovakia.} \texttt{laszlo.szalay.sopron@ gmail.com}}

\newcommand*{\TIT}{Tilings of hyperbolic $(2\times n)$-board with colored squares and dominoes}
\title{\bf \TIT}

\author{Takao Komatsu\footnote{\ADRt}, L\'aszl\'o N\'emeth\footnote{\ADRn}, L\'aszl\'o Szalay\footnote{\ADRs}}
\date{{\small \today}}

\begin{document}

\maketitle \thispagestyle{empty}

\begin{abstract}
Several articles deal with tilings with squares and dominoes of the well-known regular square mosaic in Euclidean plane, but not any with the hyperbolic regular square mosaics. In this article, we examine the tiling problem with colored squares and dominoes of one type of the possible hyperbolic generalization of $(2\times n)$-board.   \\[1mm]
 {\em Key Words: Tiling, domino, hyperbolic mosaic, Fibonacci numbers, combinatorial identity.}\\
 {\em MSC code:  Primary 05A19; Secondary 05B45, 11B37, 11B39, 52C20.} \\[1mm] 
 The final publication is available at Art Mathematica Contemporanea via http://amc-journal.eu.
\end{abstract}



\section{Introduction}

In the hyperbolic plane there exist infinite types of regular mosaics, they are denoted by Schl\"afli's symbol $\{p,q\}$, where the positive integers $p$ and $q$ have the property $(p-2)(q-2)>4$, see \cite{Cox}. If $p=4$ they are the regular square mosaics and each vertex of the mosaic is surrounded by $q$ squares.  Note that if $p=q=4$ we obtain the Euclidean square mosaic.

Now we define the $(2\times n)$-board on mosaic $\{4,q\}$, where $q\geq4$. First we take a square $S_1$ with vertices $A_0,A_1,B_1, B_0$ according to Figure~\ref{fig:board44}, and later to Figure \ref{fig:board45} and \ref{fig:board4q}. As the second step we consider the square $S_2$, which has a common edge $A_1B_1$ with $S_1$. The two new vertices are $A_2,B_2$. Similarly, we define the squares $S_3,\ldots, S_n$, their newly constructed vertices are $A_i$ and $B_i$ ($3\leq i \leq n$), respectively. The union of $S_i$ ($1\leq i \leq n$) forms the first level of the board. It is depicted with yellow colors in Figures~\ref{fig:board44}-\ref{fig:board4q}. (On the left-hand side of Figure~\ref{fig:board45} the mosaic $\{4,5\}$ and the $(2\times 4)$-board are illustrated in Poincar\'e disk model and on the right-hand side there is a ``schematic" $(2\times 4)$-board from the mosaic.) The second level of the board is formed by the squares of the mosaic having at least one vertex from the set  $\{A_1, A_2, \ldots, A_n\}$ and not from $\{B_1, B_2, \ldots, B_n,A_{n+1}\}$, where the last point is the appropriate point of the virtually joined square $S_{n+1}$ ($A_0$ is not in the first set, see  Figure~\ref{fig:board4q}). These are the light blue squares in the figures. In the first level, independently from $q$ there are $n$ squares, while the second level contains $n(q-3)$ squares (see Figure~\ref{fig:board4q}).  

Let $r_n$ be the number of the different tilings with $(1\times 1)$-squares and $(1\times 2)$-dominoes  (two squares with a common edge) of a $(2\times n)$-board of mosaic $\{4,q\}$. 
It is known that the tilings of a $(1\times n)$-board on the Euclidean square mosaic can be counted by the Fibonacci numbers \cite{BQ,BPS}. In fact, $r_n=f_n$, where $\{f_n\}_{n=0}^\infty$ is the shifted Fibonacci sequence  ($F_n=f_{n-1}$, where $F_n$ is the $n$-th Fibonacci number, A000045 in OEIS \cite{S}), so that 
$$f_{n}=f_{n-1}+f_{n-2}\quad (n\geq 2)$$ 
holds with initial values $f_0=f_1=1$ (and $f_{-1}=0$).

\begin{figure}[!h]
	\centering
	\includegraphics[scale=0.9]{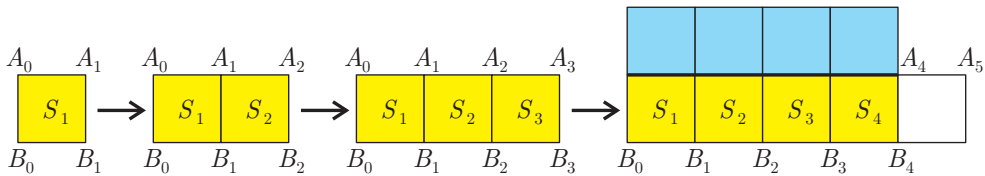}
	\caption{$(2\times 4)$-board on Euclidean mosaic $\{4,4\}$}
	\label{fig:board44}
\end{figure}

\begin{figure}[!h]
	\centering
	\includegraphics[scale=0.9]{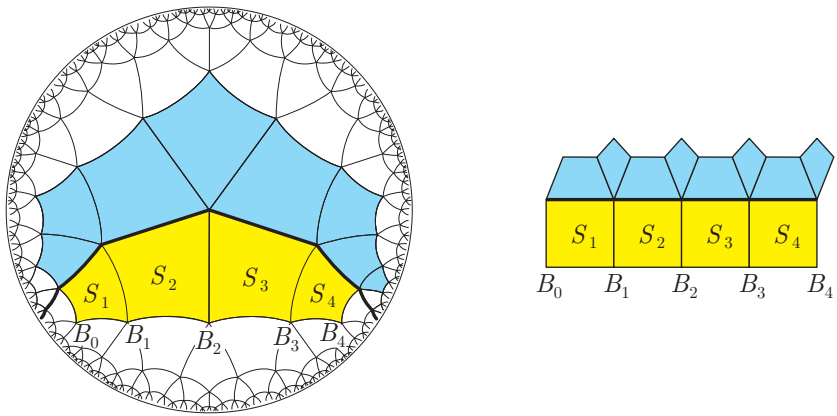}
	\caption{$(2\times 4)$-board on  hyperbolic  mosaic $\{4,5\}$}
	\label{fig:board45}
\end{figure}

\begin{figure}[!h]
	\centering
	\includegraphics[scale=0.9]{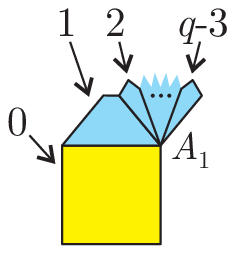}\hspace{1cm}     	\includegraphics[scale=0.9]{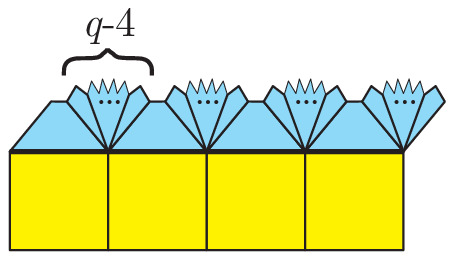}
	\caption{$(2\times 1)$-board and $(2\times 4)$-board on hyperbolic mosaic $\{4,q\}$ ($q\geq 5$)}
	\label{fig:board4q}
\end{figure}

McQuistan and Lichtman \cite{ML} (generalizations in \cite{Kah}) studied the tilings in case of the Euclidean square mosaic $\{4,4\}$  and they proved that $r_n$ satisfies the identity
\begin{equation}\label{eq:til_44}
r_n=3r_{n-1}+r_{n-2}-r_{n-3}
\end{equation}
for $n\geq3$ with initial values $r_0=1$, $r_1=2$ and $r_2=7$ (A030186 in \cite{S}).

In the work \cite{Ben}, the generalized Fibonacci number $u_n$, where
\begin{equation} \label{eq:u}
 u_{n}=a u_{n-1}+b u_{n-2},\quad (n\geq 2)
\end{equation}
 with initial values $u_0=1$, $u_1=a$ (and $u_{-1}=0$),
is interpreted as the number of ways to tile a $(1 \times n)$-board using $a$ colors of squares and $b$ colors of dominoes. Obviously, if $a=b=1$ then $u_n=f_n$. Belbachir and Belkhir proved a couple of general combinatorial identities related to $u_n $  in \cite{Belb}.

Let $R_n$ be the number of tilings of $(2\times n)$-board of mosaic $\{4,q\}$ using $a$ colors of squares and $b$ colors of dominoes.  When $q=4$ Katz and Stenson \cite{KS} showed the recurrence rule  
\begin{equation}\label{eq:tilcol_44}
R_n=(a^2+2b)R_{n-1}+a^2b\,R_{n-2}-b^3R_{n-3},\quad (n\geq3)
\end{equation}
with initial values $R_0=1$, $R_1=a^2+b$ and $R_2=a^4+4a^2b+2b^2$.

In this article, we examine the tilings of $(2\times n)$-board on mosaic $\{4,q\}$ ($q\geq4$) with colored squares and dominoes in a general way  and we obtain the following main theorem.

\begin{theorem}\label{th:maincolor}
Assume $q\geq 4$. The sequence $\{R_n\}_{n=0}^{\infty}$ can be described by  the fourth order linear homogeneous recurrence relation 
\begin{equation}\label{eq:main_color}
R_n= \alpha_q\,R_{n-1} +\beta_q\,R_{n-2}+ \gamma_q\,R_{n-3} -b^{2(q-2)} R_{n-4},\quad(n\geq4)
\end{equation}
where (explicit formulas later)
\begin{eqnarray}
	\alpha_{q+2} &=& a \alpha_{q+1}+b \alpha_{q},\label{eqs:recur_alfa}\\
	\beta_{q+3} &=& (a^2+b)\beta_{q+2}+b(a^2+b)\beta_{q+1}-b^3\beta_{q},\label{eqs:recur_beta}\\
	\gamma_{q+2} &=& -ab\gamma_{q+1}+b^3\gamma_{q} \label{eqs:recur_gamma}
\end{eqnarray}
with initial values 
\begin{eqnarray*}
	& &	\alpha_4=a^2+b,\  \alpha_5=a(a^2+3b),\\
	& &	\beta_4=2b(a^2+b),\  \beta_5=b(a^2+b)(a^2+2b),\  \beta_6=b(a^6+6a^4b+10a^2b^2+2b^3),\\
	& &	\gamma_4=b^2(a^2-b),\  \gamma_5=-ab^3(a^2+b),
\end{eqnarray*}
moreover
$R_0=1$,
$R_1=u_{q-2}$,
$R_2=u_{q-2}^2+abu_{q-4}u_{q-3}+bu_{q-3}^2+b^2u_{q-4}^2$,
$R_3=(u_{q-2}^2+2abu_{q-4}u_{q-3}+2bu_{q-3}^2+2b^2u_{q-4}^2)u_{q-2}
+b^2(u_{q-3}u_{q-4}+(a^2+b)u_{q-4}u_{q-5}
+au_{q-4}^2)u_{q-3}
+ab^3u_{q-4}^2u_{q-5}$.
\end{theorem}

If $a=b=1$, then Theorem~\ref{th:maincolor} leads to the following corollary. Recall that $f_n=F_{n+1}$ (shifted Fibonacci numbers).
\begin{corollary}\label{cor:main}
	The sequence $\{r_n\}_{n=0}^\infty$ can be given by  the fourth order linear homogeneous recurrence relation  
\begin{equation}\label{eq:main}
r_n= 2f_{q-3}\ r_{n-1} +\left(5f_{q-4}^2+(-1)^{q-1} \right)r_{n-2} 	+ 2(-1)^q f_{q-5}\ r_{n-3} - r_{n-4},\quad (n\geq4)
\end{equation}	
with initial values 
$r_0=1$, 
$r_1=f_{q-2}$, 
$r_2=7f_{q-4}^2+7f_{q-4}f_{q-5}+2f_{q-5}^2$ and 
$r_3=22f_{q-4}^3+36f_{q-4}^2f_{q-5}+19f_{q-4}f_{q-5}^2+3f_{q-5}^3$.
\end{corollary}

Observe, that if $q=4$, then  \eqref{eq:main_color} returns with \eqref{eq:tilcol_44} (compute the sum of $R_n$ and $bR_{n-1}$). Similarly, the extension of \eqref{eq:til_44} is \eqref{eq:main}.

\section{Tilings on mosaic $\{4,q\}$}

We can see that our tiling exercise of the hyperbolic $(2\times 1)$-board on the mosaic $\{4,q\}$ ($q\geq 5$) is the same as the tiling exercise of the Euclidean  $\big( 1\times (q-2)\big)$-board. So $R_1=u_{q-2}$ and $r_1=f_{q-2}$ (Figure~\ref{fig:board4q}).

Before the discussion of the main result, we define the break-ability of  a tiling.  
A tiling of a $(2\times n)$-board is breakable in position $i$ for $1\leq i\leq n-1$, if this tiling is a concatenation of the tilings of a $(2\times i)$-subboard and a $\big(2\times (n-i)\big)$-subboard. Clearly, the number of colored tilings of such a board is $R_iR_{n-i}$. A tiling is unbreakable in position $i$ in three different ways: if a domino covers  the last square of the first subboard and the first square of the second subboard either in the first or the second level, or on both levels (see Figure~\ref{fig:unbreak_in_m}).   
\begin{figure}[!h]
	\centering
	\includegraphics[scale=0.9]{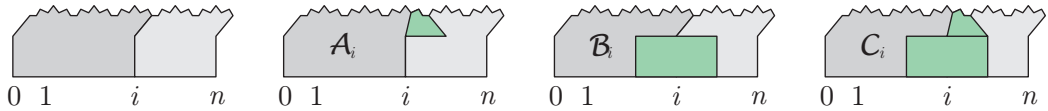}
	\caption{Breakable and unbreakable tilings in position $i$ when $q=7$}
	\label{fig:unbreak_in_m}
\end{figure}

Now, we define three subboards. Let $\mathcal{A}_i$,  $\mathcal{B}_i$ and  $\mathcal{C}_i$ be the subboards of $(2\times i)$-board ($1\leq i\leq n$), respectively, where the last square from second level,  the last square from first level and the last squares from both levels are deleted from $(2\times i)$-board. In Figure~\ref{fig:unbreak_in_m} these subboards are illustrated. Let $A_i$,  $B_i$ and  $C_i$ denote the number of  different colored tilings of $\mathcal{A}_i$,  $\mathcal{B}_i$ and  $\mathcal{C}_i$, respectively.

\subsection{Proof of Theorem \ref{th:maincolor} and Corollary \ref{cor:main}}

Our proof is based on the connections among $(2\times n)$-board, $\mathcal{A}_n$,  $\mathcal{B}_n$ and  $\mathcal{C}_n$ subboards. We can easily give the number of tilings if  $n=1$. They are  $R_1=u_{q-2}$, $A_1=u_{q-4}$, $B_1=u_{q-3}$ and $C_1=u_{q-4}$. Moreover let  $R_0=1$, $A_0=B_0=C_0=0$.

Generally, if $n\geq2$, then Figure~\ref{fig:recu_system_4q_fig} shows the recurrence connections of the subboards. For example, let us see the first row. We can build a full $(2\times n)$-board by four different ways from the full $\big(2\times (n-1)\big)$-board or from the subboards $\mathcal{A}_{n-1}$,  $\mathcal{B}_{n-1}$ and  $\mathcal{C}_{n-1}$. If we join a suitable $(2\times 1)$-board to a $\big(2\times (n-1)\big)$-board, then the coefficient $u_{q-2}$ is obvious in case of the breakable tilings in position $n-1$. When we complete $\mathcal{A}_{n-1}$ to a full $(2\times n)$-board, we have a domino in the second level with $b$ different colors, and we put a square onto the first level with $a$ colors.  
(If we replace the laid down domino in the second level with two squares, then these tilings would be a part of the first case when we completed the $\big(2\times (n-1)\big)$-board.)    
The rest part can be tiled freely. 
Consequently, the coefficient of $A_{n-1}$ is $abu_{q-4}$ and these are unbreakable tilings in position $n-1$.  
Now, let us complete $\mathcal{B}_{n-1}$ and $\mathcal{C}_{n-1}$ to be full $(2\times n)$-board with a domino in the first level or with two dominoes,  one is  in the first level and the other in the second level, respectively. The rest parts can be tiled freely. We obtain $bu_{q-3}$ and $b^2u_{q-4}$ new (unbreakable in position $n-1$)  tilings. 
Summarising the result of the first row of Figure~\ref{fig:recu_system_4q_fig} we have the first equation of the system of recurrence equations \eqref{eq:rec_system}. The determinations of the other rows can be explained similarly. We mention, that, for example, in the fourth row $\mathcal{B}_{n-1}$ does not appear, because 
when we complete it to $\mathcal{C}_{n}$ we do not have new tiling type,  the tilings are in the first tiling types in the same row. (The yellow square would be in the grey $\big(2\times (n-1)\big)$-board -- see the last row in Figure~\ref{fig:recu_system_4q_fig}.)
\begin{figure}[!h]
	\centering
	\includegraphics[scale=0.9]{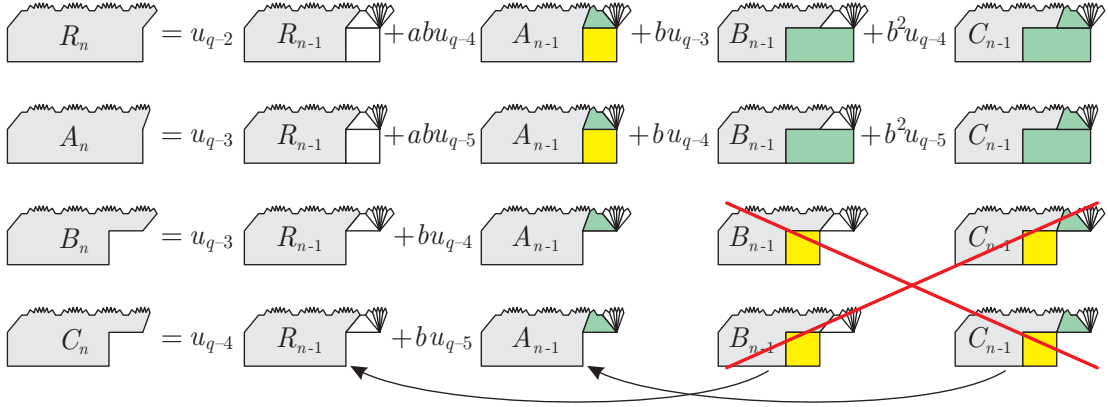}
	\caption{Base of recurrence connections of the subboards}
	\label{fig:recu_system_4q_fig}
\end{figure}

Hence the recurrence equations for $n\geq1$ satisfy the system 
\begin{equation}\label{eq:rec_system}
\begin{array}{ccl}
R_n&=& u_{q-2}R_{n-1}+ab\,u_{q-4}A_{n-1}+b\,u_{q-3}B_{n-1}+b^2\,u_{q-4}C_{n-1}\\ 
A_n&=& u_{q-3}R_{n-1}+ab\,u_{q-5}A_{n-1}+b\,u_{q-4}B_{n-1}+b^2\,u_{q-5}C_{n-1}\\
B_n&=& u_{q-3}R_{n-1}+b\,u_{q-4}A_{n-1}\\
C_n&=& u_{q-4}R_{n-1}+b\,u_{q-5}A_{n-1}
\end{array}.
\end{equation}
Recall that the initial values are $R_0=1$, $A_0=B_0=C_0=0$.

The matrix of the coefficients of \eqref{eq:rec_system} is 
$$\mathbf{M}=\begin{pmatrix}
u_{q-2} & ab\,u_{q-4} & b\,u_{q-3} & b^2\,u_{q-4}\\
u_{q-3} & ab\,u_{q-5} &b\, u_{q-4} & b^2\,u_{q-5}\\
u_{q-3} & b\,u_{q-4} & 0 & 0\\
u_{q-4} & b\,u_{q-5} & 0 & 0
\end{pmatrix}. $$
As usual, the characteristic equation of $\mathbf{M}$ provides the recurrence relation for $\{R_n\}$ (and $\{A_n\}$, $\{B_n\}$, $\{C_n\}$; see the proof in \cite{NSz_Power}. The computation was made by the help of software \textsc{Maple}.) Thus we have
\begin{equation}
R_n= \alpha_q\,R_{n-1} +\beta_q\,R_{n-2}+ \gamma_q\,R_{n-3} +\delta_q R_{n-4} \quad (n\geq4),
\end{equation}
where (with some calculation using \eqref{eq:u})
\begin{eqnarray*}
	\alpha_{q} &=& abu_{q-5}+u_{q-2},\\
	\beta_{q} &=& b(b^2u_{q-5}^2 -au_{q-5}u_{q-2} +2bu_{q-4}^2 +au_{q-4}u_{q-3} +u_{q-3}^2),\\
	\gamma_{q} &=& -b^2(bu_{q-5}^2u_{q-2}-2u_{q-4}u_{q-3}^2+au_{q-5}u_{q-3}^2 +u_{q-4}^2u_{q-2} ), \\
	\delta_{q} &=& -b^4(u_{q-5}^2 u_{q-3}^2-2u_{q-5}u_{q-4}^2 u_{q-3}+u_{q-4}^4).	
\end{eqnarray*}
Moreover, we obtain the initial values of the recurrence for $n=1,2,3$ from system \eqref{eq:rec_system}. They are $R_1=u_{q-2}$,
$R_2=u_{q-2}^2+abu_{q-4}u_{q-3}+bu_{q-3}^2+b^2u_{q-4}^2$ and
\begin{multline}
R_3=(u_{q-2}^2+abu_{q-4}u_{q-3}+bu_{q-3}^2+b^2u_{q-4}^2)u_{q-2} \\
+(abu_{q-2}u_{q-4}+a^2b^2u_{q-4}u_{q-5}+b^2u_{q-3}u_{q-4}+b^3u_{q-4}u_{q-5})u_{q-3}\\
+(bu_{q-2}u_{q-3}+ab^2u_{q-4}^2)u_{q-3}
+(b^2u_{q-2}u_{q-4}+ab^3u_{q-4}u_{q-5})u_{q-4}.
\end{multline}

In the next part, we prove that relations \eqref{eqs:recur_alfa}--\eqref{eqs:recur_gamma} hold. 
Firstly, we insert $\alpha_{q+2}$, $\alpha_{q+1}$ and $\alpha_{q}$ into  \eqref{eqs:recur_alfa} to have
\begin{eqnarray} \label{eq:alfa_q}
abu_{q-3}+u_{q} &=& a( abu_{q-4}+u_{q-1})+b( abu_{q-5}+u_{q-2}).
\end{eqnarray}
Apply \eqref{eq:u} consecutively with $n=q,q-1,\dots$ as follows. First plug $u_q$ into the equation \eqref{eq:alfa_q}, then substitute $u_{q-1}$ in the new equation, and so an. Finally, when $n=q-3$, we find that \eqref{eq:alfa_q} is an identity, so \eqref{eqs:recur_alfa} holds. If $q=4$ and $q=5$, then $\alpha_q$ provides the initial values.
The proofs of \eqref{eqs:recur_beta} and \eqref{eqs:recur_gamma} go similarly.

Finally, we show that   
$\delta_{q} = -b^{2(q-2)}$. 
For $q=4$ 
we immediately obtain 
$\delta_{4} = -b^4(u_{4-5}^2 u_{4-3}^2-2u_{4-5}u_{4-4}^2 u_{4-3}+u_{4-4}^4)= -b^{2\cdot 2}$. 
Then we consider the recurrence relation ($q\geq4$)
\begin{equation}\label{eq:xq}
	x^{q+1} = b^2 x^{q}.
\end{equation}
Some calculations show that both  expressions ($\delta_{q}$ and $-b^{2(q-2)}$) satisfies recursion \eqref{eq:xq}, which implies the equality.

We express the values by $u_{q-4}$ and $u_{q-5}$ by using relation \eqref{eq:u}. Thus we have
\begin{eqnarray*}
	\alpha_{q} &=& (a^2+b)u_{q-4}+2abu_{q-5},\\
	\beta_{q}  &=& (2a^2+2b)bu_{q-4}^2+(-a^3+2ab)bu_{q-4}u_{q-5}+(-a^2b+2b^2)bu_{q-5}^2,\\
	\gamma_{q} &=& (a^2-b)b^2u_{q-4}^3-(a^3-3ab)b^2u_{q-4}^2u_{q-5}-(3a^2b-b^2)b^2u_{q-4}u_{q-5}^2-2ab^4u_{q-5}^3, \\
	\delta_{q} &=& -b^{2(q-2)}.	
\end{eqnarray*}
As $F_{n}^2-F_{n}F_{n-1}-F_{n-1}^2=(-1)^{n-1}$, if $a=b=1$, then we obtain
\begin{eqnarray*}
	\alpha_{q} &=& 2f_{q-4}+2f_{q-5}=2f_{q-3},\\
	\beta_{q}  &=& 4f_{q-4}^2+f_{q-4}f_{q-5}+f_{q-5}^2=5f_{q-4}^2+(-1)^{q-1},\\
	\gamma_{q} &=&2f_{q-4}^2f_{q-5}-2f_{q-4}f_{q-5}^2-2f_{q-5}^3 = 2(-1)^qf_{q-5}, \\
	\delta_{q} &=& -1.	
\end{eqnarray*}

Now the initial values $R_i$ lead to the initial values $r_i$ ($i=1,2,3$).

\subsection{Unbreakable tilings}

In this subsection we determine the number of unbreakable tilings. Let $\widetilde{r}_n$ (and $\widetilde{R}_n$) be the number of  different unbreakable tilings with (colored) squares and dominoes of $(2\times n)$-board of $\{4,q\}$.
Moreover, let $\widetilde{A}_i$,  $\widetilde{B}_i$ and  $\widetilde{C}_i$ denote the number of the different unbreakable colored tilings of  $\mathcal{A}_i$,  $\mathcal{B}_i$ and  $\mathcal{C}_i$, respectively. 

\begin{theorem}\label{th:tilcol_rec_unbreak}
	The sequence $\{\widetilde{R}_n\}$ can be described by  the binary recurrence relation 
\begin{equation*}
\widetilde{R}_n = abu_{q-5}\widetilde{R}_{n-1}
+b^2\left(u_{q-4}^2+bu_{q-5}^2 \right)\widetilde{R}_{n-2},\quad (n\geq3) 
\end{equation*}
where the initial values are $\widetilde{R}_1=u_{q-2}$ and $\widetilde{R}_2=abu_{q-3}u_{q-4}+bu_{q-3}^2+b^2u_{q-4}^2$.
\end{theorem}
\begin{proof}
The proof is similar to the proof of the first theorem. By deleting the breakable tilings from Figure~\ref{fig:recu_system_4q_fig} (the second column) we gain the system of recurrence sequences ($n\geq2$)
\begin{eqnarray*}
	\widetilde{R}_n&=& ab u_{q-4}\widetilde{A}_{n-1} +bu_{q-3}\widetilde{B}_{n-1} +b^2u_{q-4}\widetilde{C}_{n-1}\\ 
	\widetilde{A}_n&=& ab u_{q-5}\widetilde{A}_{n-1} +bu_{q-4}\widetilde{B}_{n-1} +b^2u_{q-5}\widetilde{C}_{n-1}\\
	\widetilde{B}_n&=& b u_{q-4}\widetilde{A}_{n-1}\\
	\widetilde{C}_n&=& b u_{q-5}\widetilde{A}_{n-1}
\end{eqnarray*}
with initial values  $\widetilde{R}_1=u_{q-2}$, $\widetilde{A}_1=u_{q-3}$, $\widetilde{B}_1=u_{q-3}$, $\widetilde{C}_1=u_{q-4}$.
The characteristic equation of its coefficients matrix gives the recurrence for $\widetilde{R}_n$.  
From the system of recurrence sequences we gain
$\widetilde{R}_2$.
\end{proof}

Supposing $a=b=1$, together with 
\begin{eqnarray*}
\widetilde{r}_2 &=& 3f_{q-4}^2+3f_{q-4}f_{q-5}+f_{q-5}^2 
  = 4f_{q-4}^2+2f_{q-4}f_{q-5}+(-1)^{q-1} \\ 
  &=& 2f_{q-4}(2f_{q-4}+f_{q-5})+(-1)^{q-1},
\end{eqnarray*} 
  we obtain the following corollary.
\begin{corollary}\label{cor:til_rec_unbreak}
	The sequence $\{\widetilde{r}_n\}$ satisfies the binary recurrence relation 
	\begin{equation*}
	\widetilde{r}_n = f_{q-5}\widetilde{r}_{n-1}
	+\left(f_{q-4}^2+f_{q-5}^2 \right)\widetilde{r}_{n-2},\quad (n\ge3)
	\end{equation*}
	with coefficients linked to Fibonacci numbers, where the initial values are $\widetilde{r}_1=f_{q-2}$ and $\widetilde{r}_2=2f_{q-4}f_{q-2}+(-1)^{q-1}$.
\end{corollary}

\section{Some identities}

In the sequel, we give certain identities related to the sequences $\{R_n\}$ and $\{\widetilde{R}_{n}\}$. The proofs are based on the tilings, not on the recursive formulae.

\begin{identity} \label{id:RR}
	If $n\geq1$, then
   \begin{equation*}	
      R_n=\sum_{i=0}^{n-1}R_i\widetilde{R}_{n-i}.	
   \end{equation*}
\end{identity}
\begin{proof}
Let us consider the breakable colored tilings in position $i$ ($0\leq i <n$) of $(2\times n)$-board, where the tilings on the right $\big(2\times (n-i)\big)$-subboard are unbreakable (see Figure~\ref{fig:identity_RR}). The number of this tilings is $R_i\widetilde{R}_{n-i}$. If $i=0$, then the tilings are unbreakable on the whole $(2\times n)$-board. Clearly, when $i$ goes from 1 to $n-1$, we have different tiling and we consider all of them exhaustedly.  
\end{proof}
\begin{figure}[!h]
	\centering
	\includegraphics[scale=0.9]{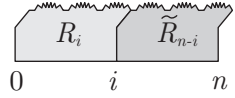}
	\caption{Breakable tilings in position  $i$ in case of Identity \ref{id:RR}}
	\label{fig:identity_RR}
\end{figure}

An equivalent form of Identity \ref{id:RR} is
\begin{identity} \label{id:RR2}
	If $n\geq1$, then
	\begin{equation*}	
	R_n=\sum_{i=1}^{n}R_{n-i}\widetilde{R}_{i}.	
	\end{equation*}
\end{identity}

The next statement gives another rule of summation.

\begin{identity} \label{id:RRR}
  If $m\geq1$ and $n\geq1$, then
  \begin{equation*}	
  R_{n+m}=R_{n}R_{m}+\sum_{i=1}^{n}\sum_{j=1}^{m}R_{n-i}R_{m-j}\widetilde{R}_{i+j}.	
  \end{equation*}
\end{identity}
\begin{proof}
Let us consider a $\big(2\times (n+m)\big)$-board as the concatenation of $(2\times n)$-board and $(2\times m)$-board (in other words, tilings are breakable in position $n$).	First we take the breakable tilings in position $n$, their cardinality is $R_nR_m$. Then we examine the unbreakable tilings in this position. We cover the position $n$ by $i+j$ long unbreakable tilings from position $n-i$ to $n+j$. They give the rest tilings. Figure~\ref{fig:identity_RRR} illustrates these two cases.
\end{proof}
\begin{figure}[!h]
	\centering
	\includegraphics[scale=0.9]{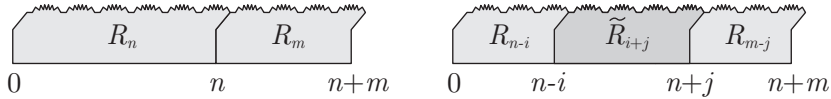}
	\caption{Tilings in case of Identity \ref{id:RRR}}
	\label{fig:identity_RRR}
\end{figure}

Identity \ref{id:RRR} admits the following remarkable specific cases by the choice of $m=1$, $m=(k-1)n$ and $n=n-k$, $m=n+k$, respectively.
\begin{identity} \label{id:RRR2}
	If  $n\geq1$, then
	\begin{equation*}	
	R_{n+1}=R_{n}R_{1}+\sum_{i=1}^{n}R_{n-i}\widetilde{R}_{i+1}.	
	\end{equation*}
\end{identity}

\begin{identity} \label{id:RRR3}
	If $n\geq1$  and $k\geq2$, then
	\begin{equation*}	
      R_{kn}=R_{n}R_{(k-1)n}+\sum_{i=1}^{n}\sum_{j=1}^{(k-1)n}R_{n-i}R_{(k-1)n-j}\widetilde{R}_{i+j}.	
	\end{equation*}
\end{identity}

\begin{identity} \label{id:RRR4}
	If     $n > k\geq0$ then
	\begin{equation*}	
	R_{2n}=R_{n-k}R_{n+k}+\sum_{i=1}^{n-k}\sum_{j=1}^{n+k}R_{n-k-i}R_{n+k-j}\widetilde{R}_{i+j}.	
	\end{equation*}
\end{identity}

Finally, we give an identity about the product of two arbitrary terms of the sequence $\{R_n\}$.

\begin{identity}\label{id:RRRR}
	If $n, m\geq1$, then
	$$R_{n}R_{m}  = \sum_{i=0}^{n-1}\sum_{j=0}^{m-1} R_{i}R_{j}\widetilde{R}_{n-i}\widetilde{R}_{m-j}.$$
\end{identity}
\begin{proof}
Consider a $\big(2\times (n+m)\big)$-board as a concatenation of  $(2\times n)$-board and $(2\times m)$-board.	The result is derived in a direct manner from the number of the breakable tilings in position $n$. See Figure~\ref{fig:identity_RRRR}.
\end{proof}
\begin{figure}[!h]
	\centering
	\includegraphics[scale=0.9]{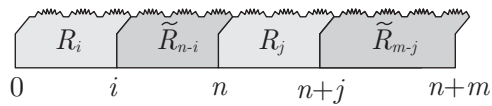}
	\caption{Tilings in case of Identity \ref{id:RRRR}}
	\label{fig:identity_RRRR}
\end{figure}

\section{Conclusion and future work}

In this article, we introduced a generalization of the square boards on the hyperbolic regular square mosaics and examined the combinatorial properties of tilings on these mosaics with colored squares and dominoes. As there are the infinite number of regular mosaics in the hyperbolic plane we hope that the examinations of the combinatorial properties of other tilings give some useful results. Moreover, we are informed on two additional timely articles about hyperbolic space tilings \cite{Lucis,Prok}.


\begin{thebibliography}{29}
	
\bibitem{Belb} H. Belbachir and A. Belkhir, {Tiling approach to obtain identities for generalized Fibonacci and Lucas numbers}, {\it Ann. Math. Inf.}, \textbf{41} (2013), 13--17.

\bibitem{Ben} A. T. Benjamin and J. J. Quinn, {The Fibonacci numbers -- exposed more discretely}, {\it Math. Magazine}, \textbf{33} (2002), 182--192.	
	
\bibitem{BQ} A. T. Benjamin and J. J. Quinn, {\it Proofs that really count: The art of combinatorial proof}, Mathematical Association of America, 2003.

\bibitem{BPS} A. T. Benjamin, S. S. Plott and J. Sellers, 
Tiling proofs of recent sum identities involving Pell numbers, 
{\it Ann. Comb.} {\bf 12} (2008), 271--278.

	
\bibitem{Cox} H. S. M. Coxeter,  {Regular honeycombs in hyperbolic space}, \textit{Proc.~Int.~Congress Math.} \textbf{III} (1954), 155-169.
	
\bibitem{Kah} R. Kahkeshani, The tilings of a $(2 \times n)$-board and some new combinatorial identities, \textit{J. Integer Seq.} \textbf{20} (2017),  {Article 17.5.4}.
	

\bibitem{KS} M. Katz and C. Stenson, Tiling a $(2\times n)$-board with squares and dominoes, {\it J. Integer Seq.} {\bf 12} (2009), Article 9.2.2.

\bibitem{Lucis} Z. Lu\v{c}i\'c, E. Moln\'ar and N. Vasiljevi\'c, An algorithm for classification of fundamental polygons for a plane discontinuous group, 
S\textit{pringer Contributed Volume dedicated to K\'aroly Bezdek and Egon Schulte to Their 60th Birthday}, 
Conference GEOSYM, Veszprém/Hungary 2015, Ed.s: A. Deza, M. Conder and A. Ivi?-Weiss, Springer, 2018.

\bibitem{ML} R. B. McQuistn and S. J. Lichtman, Exact recursion relation for $(2\times N)$ arrays of dumbbells, {\it J. Math Phys.} {\bf 11} (1970), 3095--3099.

\bibitem{NSz_Power}  L. N\'emeth and L. Szalay,   Power sums in hyperbolic Pascal triangles. \emph{Analele Stiintifice ale Universitatii Ovidius, Seria Matematica}, \textbf{26}(1) (2018),  (to appear).

\bibitem{Prok} I. Prok, On maximal homogeneous 3-geometries -- a polyhedron algorithm for space tilings, \textit{Universe}, MDPI Basel, Switzerland (2017), (to appear).

\bibitem{S} N. J. A. Sloane, The On-Line Encyclopedia of Integer Sequences, 
http://oeis.org.
	

\end{thebibliography}
\end{document}